\documentclass[review,hidelinks,onefignum,onetabnum]{siamart220329}

\usepackage{lipsum}
\usepackage{amsfonts}
\usepackage{graphicx}
\usepackage{epstopdf}
\usepackage{algorithmic}
\usepackage{esint}
\usepackage{subfigure}
\usepackage{multirow}
\ifpdf
  \DeclareGraphicsExtensions{.eps,.pdf,.png,.jpg}
\else
  \DeclareGraphicsExtensions{.eps}
\fi

\newsiamremark{remark}{Remark}
\newsiamremark{hypothesis}{Hypothesis}
\crefname{hypothesis}{Hypothesis}{Hypotheses}
\newsiamthm{claim}{Claim}

\usepackage{amsfonts,amssymb,amsmath}
\newcommand{\tcb}[1]{\textcolor{blue}{#1}}

\newcommand{\E}{\mathbb{E}}
\newcommand{\R}{\mathbb{R}}

\def\b{\boldsymbol}

\def\e{\epsilon}

\def\cA{\mathcal{A}}
\def\cT{\mathcal{T}}
\def\cH{\mathcal{H}}
\def\cF{\mathcal{F}}

\DeclareMathOperator{\tr}{tr}

\usepackage{amsopn}

\numberwithin{equation}{section}

\title{Convergence analysis of the Random ordinate method for mitigating the ray effect
\thanks{\today
\funding{Lei Li is partially supported by the National Key R\&D Program of China No. 2020YFA0712000; Shanghai Municipal Science and Technology Major Project 2021SHZDZX0102, NSFC 12371400. Min Tang is funded by The Strategic Priority Research Program of Chinese Academy of Sciences, No.XDA25010401, NSFC12411530067, NSFC12031013, and partially by Mevion Medical Equipment Co., Ltd.}}}

\author{Lei Li\thanks{School of Mathematical Sciences, Institute of Natural Sciences, MOE-LSC,
Shanghai Jiao Tong University, Shanghai, P.R. China.
  (\email{leili2010@sjtu.edu.cn} 
  ).}
\and Min Tang\thanks{School of Mathematical Sciences,, Institute of Natural Sciences and MOE-LSC,
Shanghai Jiao Tong University, Shanghai, P.R. China. 
  (\email{tangmin@sjtu.edu.cn}).}
\and Yuqi Yang\thanks{School of Mathematical Sciences,, Institute of Natural Sciences, Shanghai Jiao Tong University, Shanghai, P.R. China. 
  (\email{yangyq1@sjtu.edu.cn}).}}
\ifpdf
\hypersetup{
  pdftitle={},
  pdfauthor={}
}
\fi

\begin{document}

\maketitle

\begin{abstract}
 The Discrete Ordinates Method (DOM) is widely used for velocity discretization in radiative transport simulations. However, DOM tends to exhibit the ray effect when the velocity discretization is not sufficiently refined, a limitation that is well documented. To counter this, we have developed the Random Ordinates Method (ROM) by integrating randomness into the velocity discretization, which mitigates the ray effect without incurring additional computational costs.
ROM partitions the velocity space into $ n $ cells, selects a random ordinate from each cell, and solves a DOM system with these ordinates. It leverages the average of multiple samples to achieve a higher convergence order, especially for solutions with low regularity in the velocity variable. In this work, we provide a detailed convergence analysis for ROM, focusing on bias and single-run errors. This analysis is crucial for determining the necessary mesh size and the optimal number of samples required to attain a specified level of accuracy.

\end{abstract}



\section{Introduction}
The Radiative Transfer Equation (RTE) has been widely applied in various fields, including astrophysics, fusion research, and biomedical optics. The steady-state RTE with anisotropic scattering is given by:
\begin{equation}\label{eq:equ_1}
\mathbf{u} \cdot \nabla \psi(\mathbf{z}, \mathbf{u}) + \sigma_{T}(\mathbf{z}) \psi(\mathbf{z}, \mathbf{u}) = \sigma_{S}(\mathbf{z}) \fint_S P(\mathbf{u}',\mathbf{u})\psi(\mathbf{z}, \mathbf{u}') \mathrm{d} \mathbf{u}' + q(\mathbf{z}),
\end{equation}
subject to the inflow boundary condition:
\begin{equation}\label{eq:equ_1_bc}
\psi(\mathbf{z}, \mathbf{u}) = \psi_{\Gamma}^{-}(\mathbf{z}, \mathbf{u}), \quad \mathbf{z} \in \Gamma^{-} = \partial \Omega, \quad \mathbf{u} \cdot \mathbf{n}_{\mathbf{z}} < 0.
\end{equation}
Here, $\psi(\mathbf{z}, \mathbf{u})$ denotes the probability density function of particles moving in direction $\mathbf{u}$ at position $\mathbf{z}$. $\mathbf{z} \in \Omega \subset \mathbb{R}^3$ represents the spatial variable; $\mathbf{u}$ denotes the particle's direction of motion, and $S = \{\mathbf{u} \mid \mathbf{u} \in \mathbb{R}^3, |\mathbf{u}| = 1 \}$; $\fint_S := \frac{1}{|S|}\int_S$; $\mathbf{n}_{\mathbf{z}}$ denotes the outward normal vector at position $\mathbf{z}$. The coefficients $\sigma_{T}(\mathbf{z})$, $\sigma_{S}(\mathbf{z})$, and $q(\mathbf{z})$ represent the total cross-section, scattering cross-section, and source term, respectively. For physically meaningful scenarios, $\sigma_{S}(\mathbf{z})$ and $\sigma_{T}(\mathbf{z})$ are strictly positive, bounded functions, while $q(\mathbf{z})$ is nonnegative and bounded. Furthermore, $\sigma_{S}(\mathbf{z})/\sigma_{T}(\mathbf{z}) \leq 1$ for all $\mathbf{z} \in \Omega$.
The scattering phase function $P(\mathbf{u}',\mathbf{u})$ satisfies:
$$
P(\mathbf{u}',\mathbf{u}) = P(\mathbf{u},\mathbf{u}'), \qquad \fint_S P(\mathbf{u}',\mathbf{u})\mathrm{d} \mathbf{u} = \frac{1}{|S|}\int_S P(\mathbf{u}',\mathbf{u})\mathrm{d} \mathbf{u} = 1,
$$
where $\frac{1}{|S|}P(\mathbf{u}',\mathbf{u})$ gives the probability that particles moving in direction $\mathbf{u}'$ are scattered into direction $\mathbf{u}$.

Classical numerical algorithms for solving the RTE can be broadly categorized into two groups: particle-based methods and PDE-based methods. 
Particle-based methods primarily refer to the well-known Monte Carlo (MC) methods \cite{bhan2007condensed,1971An}, which simulate the trajectories of a large number of particles. In contrast, PDE-based methods utilize partial differential equation (PDE) solvers and require meshes for both spatial and velocity variables. Various spatial discretization techniques have been developed, including the finite difference method (FDM) \cite{1992FDMFDM}, the finite volume method (FVM) \cite{Eymard2000}, and the discontinuous Galerkin method (DGM) \cite{Jin1996}. 
Compared to the extensive research on spatial discretization, velocity discretization has received relatively less attention. Among the existing methods, the Discrete Ordinates Method (DOM) \cite{1970Nuclear, Carlson1955SOLUTION} is the most widely adopted approach for velocity space discretization.
 
 The DOM discretizes the velocity space $S$ into $N$ discrete directions and approximates the integral on the right-hand side of \eqref{eq:equ_1} using a weighted sum over these directions. Specifically, by denoting $\psi_j(\mathbf{z}) \approx \psi(\mathbf{z}, \mathbf{u}_j)$, the DOM solves the following system:
\begin{equation}\label{eq:DOMquad}
\mathbf{u}_j \cdot \nabla \psi_j(\mathbf{z}) + \sigma_{T}(\mathbf{z}) \psi_j(\mathbf{z}) = \sigma_{S}(\mathbf{z}) \sum_{i=1}^{N} \omega_{i} P_{j,i} \psi(\mathbf{z}, \mathbf{u}_{i}) + q(\mathbf{z}),
\end{equation}
where $N$ is the number of discrete ordinates, $\omega_i$ ($i = 1, \dots, N$) are the corresponding weights, $\{\mathbf{u}_i, \omega_i\}$ constitutes the quadrature set, and $P_{j,i} \approx P(\mathbf{u}_i, \mathbf{u}_j)$. Classical choices for the quadrature set $\{\mathbf{u}_i, \omega_i\}$ include uniform and Gaussian quadratures, as discussed in \cite{Larsen2010}. 

The DOM provides accurate approximations for sufficiently large $N$. However, when the number of ordinates is limited, the total density $\fint_S \psi(\mathbf{u}) \, \mathrm{d}\mathbf{u}$ exhibits nonphysical spatial oscillations known as ray effects \cite{2013MomentClosures, LathropRayeffect, Larsen2010}. These artifacts arise because particles are constrained to propagate only along predefined directions. Neither careful selection of the quadrature set $\{\mathbf{u}_i, \omega_i\}$ nor spatial mesh refinement can effectively mitigate ray effects \cite{LathropRayeffect}. Instead, resolving these issues requires increasing the number of ordinates, which significantly escalates computational costs.

Numerical experiments in \cite{li2024randomordinatemethodmitigating} demonstrate that for benchmark problems exhibiting ray effects, the solution exhibits sharp transitions in the velocity variable within phase space. Consequently, when the number of discrete ordinates is insufficient to resolve these variations, the DOM exhibits low convergence rates. Specifically, in slab geometry with Lipschitz continuous inflow boundary conditions, the convergence rates of both uniform and Gaussian quadratures are approximately $3/2$. In X-Y geometry, numerical results for the line source problem—a widely used benchmark for ray effects \cite{2013MomentClosures}—show the convergence rate decreases to $0.75$.

To address the longstanding challenge of ray effects and improve the convergence rate of the DOM for solutions with low regularity, we proposed the Random Ordinates Method (ROM) in \cite{li2024randomordinatemethodmitigating}, which introduces randomness into the DOM framework. The core idea of ROM stems from the fact that randomized methods can enhance the convergence order of numerical integration for low-regularity functions \cite{Novak1988}. While deterministic quadratures may exhibit poor performance for certain Lipschitz integrands, randomized quadratures can achieve better average-case behavior by leveraging statistical properties.

Specifically, ROM partitions the velocity space $S$ into $n$ cells denoted as $S_{\ell}$ ($\ell = 1, \dots, n$). Within each cell $S_{\ell}$, a direction $\mathbf{u}_{\ell}$ is sampled uniformly at random. The ROM then solves the following system:
\begin{equation}\label{eq:ROMquad}
\mathbf{u}^\xi_\ell \cdot \nabla \psi^\xi_\ell(\mathbf{z}) + \sigma_{T}(\mathbf{z}) \psi^\xi_\ell(\mathbf{z}) = \sigma_{S}(\mathbf{z}) \sum_{\ell'=1}^{n} \omega^\xi_{\ell'} P^\xi_{\ell,\ell'} \psi^\xi(\mathbf{z}, \mathbf{u}_{\ell'}) + q(\mathbf{z}),
\end{equation}
where $\mathbb{V}^\xi = \{\mathbf{u}^\xi_{1}, \dots, \mathbf{u}^\xi_{n}\}$ denotes the $\xi$-th random sample of directions, $\omega^\xi_{\ell'} = \frac{|S_{\ell'}|}{|S|}$, and $P^\xi_{\ell, \ell'} = P(\mathbf{u}^\xi_{\ell'}, \mathbf{u}^\xi_\ell)$. By solving multiple instances of \eqref{eq:ROMquad} and averaging the results, ROM mitigates ray effects while requiring minimal modifications to existing DOM implementations. This approach offers several advantages: it preserves the structure of DOM-based codes, facilitates parallelization, and effectively reduces ray effects through statistical averaging.

To demonstrate that the computational cost of ROM does not increase compared to classical DOM when achieving the same accuracy (as numerically shown in \cite{li2024randomordinatemethodmitigating}), one must show the convergence order of both the expectation of all samples and the variance of individual samples for ROM. As numerically demonstrated in \cite{li2024randomordinatemethodmitigating}, in slab geometry, the convergence orders of uniform and Gaussian quadratures under Lipschitz continuous inflow boundary conditions are $3/2$, while that of ROM is $3$. For tests with a localized, compactly supported source term in X-Y geometry, the corresponding numerical convergence orders are $0.75$ and $1.5$, respectively. These numerically observed convergence orders are critical for determining the required mesh size (or number of ordinates) and the number of samples needed to achieve a specified accuracy. Thus, they are essential for validating the advantages of ROM and have been numerically confirmed in both slab and X-Y geometries in \cite{li2024randomordinatemethodmitigating}.

More precisely, suppose numerical observations show that the convergence order of a typical ROM sample is $\alpha \leq 1$, while the expectation of multiple samples converges at order $2\alpha$. For benchmark tests where the solution exhibits low regularity in the velocity variable, DOM has a convergence order comparable to that of a single ROM sample. Thus, when using $N$ ordinates, DOM achieves an accuracy of $O(N^{-\alpha})$.  
In contrast, since the expectation of ROM converges at twice the order of DOM, each ROM sample requires at most $\sqrt{N}$ ordinates to achieve an expected accuracy of $O((\sqrt{N})^{-2\alpha}) = O(N^{-\alpha})$, matching that of DOM. If the variance of ROM samples with $\sqrt{N}$ ordinates is controlled by $O((\sqrt{N})^{-\alpha}) = O(N^{-\alpha/2})$, by the central limit theorem, averaging $O(N^{\alpha})$ such samples reduces the variance to $O(N^{-\alpha})$ (and thus the error to $O(N^{-\alpha})$).  
For example, if $\alpha = 1$, achieving the same accuracy requires ROM to use $\sqrt{N}$ ordinates with $O(N)$ samples, resulting in a computational cost comparable to DOM with $N$ ordinates. While both ROM and DOM incur $O(N^2)$ computational cost, ROM uses a total of $O(N \cdot \sqrt{N}) = O(N^{3/2})$ ordinates, thereby mitigating ray effects. Hence, ROM emerges as a promising alternative for RTE simulations when the solution regularity in the velocity variable is low.

In this work, we rigorously prove that for slab geometry, a single realization of ROM achieves a convergence order of $3/2$, while the expectation of multiple realizations attains third-order convergence. To avoid singularities at $\mu = 0$, we truncate the velocity space from $[-1,1]$ to $[-1,-\delta) \cup (\delta,1]$, where $\delta$ is a small positive constant. We then expand the solutions of DOM and ROM into power series of the operators $\mathcal{T}$ and $\mathcal{T}^{\xi}$ (defined in \eqref{eq:T} and \eqref{eq:T_xi}), respectively. The discrepancy $\delta \mathcal{T}^{\xi} = \mathcal{T} - \mathcal{T}^{\xi}$ is then analyzed to estimate the convergence rate and variance.  In particular, the paper proves a Rosenthal-type inequality for operators, which is essential for controlling $E(\|\delta \mathcal{T}^{\xi}\|^2)$. This is the most difficult part of the proof because it involves the norm concentration problem of random operators in Hilbert spaces.

Various strategies have been proposed to address ray effects. Lathrop first proposed the first-collision source method in \cite{RemediesforRayEffects} and implemented it in the two-dimensional transport program TWOTRAN. A goal-oriented regional angular adaptive algorithm is proposed in \cite{Goal-OrientedRegional}, and rotated quadrature methods are discussed in \cite{Ray_effect_mitigation_for_the_discrete_ordinates_method_through_quadrature_rotation, Artificial_Scattering, MorelAnalysisofRayEffect, Frame_Rotation}; their main idea involves using biased or rotated quadratures. However, most works in the literature demonstrate the effectiveness of different methods in mitigating ray effects through numerical tests.

Rigorous theoretical proofs for ray effect mitigation strategies are rare. Our approach provides a general framework for analyzing randomized methods coupled with complex operators. Extending the current work to higher-dimensional cases would be of particular interest, as it would provide the first rigorous proof of a ray effect mitigation strategy. Although the current work focuses on slab geometry, for higher-dimensional cases, we expect a similar framework, albeit with more complex operator details.
The main contributions of this paper are as follows:
\begin{itemize}
    \item Theoretical results on the convergence order of velocity discretizations for irregular inflow boundary conditions are scarce in the literature. Here, we rigorously establish the convergence order of the DOM for Lipschitz continuous inflow boundary conditions.
    \item The convergence order of the expectation of ROM and the variance of ROM samples are rigorously proved for slab geometry. These results provide crucial parameters for the implementation of ROM: the number of ordinates and samples required to achieve a given accuracy.
\end{itemize}

This paper is structured as follows: In Section \ref{sec:Model}, details of the RTE with isotropic scattering in slab geometry and the expansion of the solution are described. The main theorem and the framework of the analytical proof for the convergence orders of bias and single-run error are presented in Section \ref{sec:result}. The proofs of some important lemmas are given in Section \ref{sec:usefulestimates}, and we conclude with some discussions in Section \ref{sec:discussion}.

\section{The model and problem setup} \label{sec:Model}

In this section, we introduce the details of the regularized RTE model for isotropic scattering in slab geometry. While the original RTE model involves a spatial variable $\boldsymbol{z} \in \Omega \subset \mathbb{R}^3$ and a direction $\boldsymbol{u} \in S$, the slab geometry RTE is a standard simplified model that is both physically meaningful and analytically tractable. Most analytical and numerical studies of the RTE begin with the slab geometry case. A detailed derivation of the slab geometry equation from the 3D RTE is provided in Appendix \ref{app:reduction}. We also present the expansion of solutions into convergent sequences, based on which the error and bias of ROM are estimated.

\subsection{The regularized model in the slab geometry and 
ROM}

The RTE in slab geometry with isotropic scattering kernel (i.e., $P(\mu',\mu)=1$ and $|S|=2$) writes
\begin{equation}\label{eq:1D_iso}
\mu \partial_x \psi(x,\mu)+\sigma_{T}(x) \psi(x,\mu)=\sigma_{S}(x) \phi(x)+ q(x),
\end{equation}
where $x\in  \Omega=[x_L,x_R]$ and
\begin{gather}
\phi(x)=\mathcal{I}(\psi)(x):=\fint_{S}\psi(x, \mu)\,d\mu=\frac{1}{2}\int_{-1}^1 \psi(x,\mu)d\mu,
\end{gather} 
subject to the inflow boundary conditions 
\begin{equation}\label{eq:1D_bound}
\psi(x_L, \mu)=\psi_L(\mu),\quad \mu>0;\qquad \psi(x_R, \mu)=\psi_R(\mu), \quad\mu<0.
\end{equation}

In this paper we consider the case when
$\sigma_S (x)/\sigma_T (x)$ is away from $1$, which is physically relevant in most applications. In particular, we define
\begin{gather}
\lambda=\left\|\frac{\sigma_S(x)}{\sigma_T(x)}\right\|_{\infty}\in (0, 1),
\end{gather}
and then the equation \eqref{eq:1D_iso} can be rewritten into
\begin{equation}\label{eq:1D_analysisequation}
    \mu\partial_x \psi(x,\mu)+\sigma_{T}(x)\psi(x,
    \mu)=\lambda \sigma_{r}(x)\phi(x)+q(x),
\end{equation}
where  
\[
\sigma_r(x)=\frac{\sigma_S(x)}{\|\sigma_S(x)/\sigma_T(x)\|_{\infty}}\le \sigma_T(x).
\]

Similar as in \cite{GOLSE1988110}, the singularity near $\mu=0$ will affect the estimates of the order of convergence. We will thus consider a truncated approximation similar to Grad's angular cutoff for the Boltzmann equation. In particular, we take $\delta\in(0,1)$
and consider the truncated velocity space $S^{\delta}=[-1,-\delta)\cup (\delta, 1]$. Then, we regularize the equation by redefining $\phi$ to be
\begin{gather}\label{eq:regularizedsource}
\phi(x)=\mathcal{I}^{\delta}(\psi)(x)=\fint_{S^{\delta}}\psi(x,\mu)\,d\mu.
\end{gather}
The difference between this regularized systems and the original systems should be small if $\delta$ is small. A simple error estimate of this regularization is provided in Appendix \ref{app:regerr}, which is done by the stability of the transport operators using energy type estimate. Careful estimate of this regularization error is not our focus and we skip. From a practical viewpoint, this truncated system makes sense since the numerical velocities will not touch $\mu=0$.  Below, for the convenience of notations, we will drop the super-index $\delta$.
We will then analyze the convergence of ROM for this regularized model, showing its benefit and find its optimal convergence rate.

According to the setting in \cite{li2024randomordinatemethodmitigating}, we divide $[-1, -\delta)$ and $(\delta, 1]$ into $m=n/2$ ($n$ is odd) subintervals symmetrically and each subinterval can be defined $S_{\ell},\ell=1,2,\cdots,m,m+1,\cdots,n$. For the ROM approximation, each random ordinates $\mu_{\ell}$ ($1\le \ell \le m$) is chosen randomly from $S_\ell$.  $\mu_{\ell+m}$ for the other half $(\delta, 1]$ are chosen in a symmetric fashion. 
 Then, let the weight be given by 
\begin{gather}\label{omegal}
\omega_{\ell}=\frac{|S_{\ell}|}{|S|}.
\end{gather}
Since $S_\ell$ ($\ell = 1, 2, \cdots, m, m+1, \cdots, n$) forms a partition of $S$, we have $\sum_{\ell=1}^n \omega_{\ell} = 1$.
One run of ROM is then to solve
\begin{equation}\label{eq:ROMslab}
 \mu_{\ell}^\xi\frac{d \psi^\xi_{\ell}(x)}{d x}+\sigma_{T}(x)\psi^\xi_{\ell}(x)=\lambda\sigma_{r}(x)\phi^{\xi}(x)+q(x),\quad \mu_{\ell}^\xi \in \mathbb{V}^{\xi},
\end{equation}
with
\[
\phi^{\xi}(x)=\sum_{\ell'\in V^{\xi}}\omega_{\ell'}\psi^\xi_{\ell'}(x)=\frac{1}{n}\sum_{\ell'=1}^{n}\alpha_{\ell'}\psi^\xi_{\ell'}(x).
\]
Here the superscript $\xi$ denotes the $\xi$th sample and $\mathbb{V}^{\xi}$ is the set of randomly chosen ordinates with $\mu^\xi_\ell\in S_\ell$. Moreover, since the magnitude of $\omega_{\ell}$ relates to the mesh size, for the convenience of estimating the convergence order, we introduce the rescaled weights
\begin{equation}\label{alphal}
 \alpha_{\ell}=n \cdot  \omega_{\ell}, 
\end{equation}
so that $\sum_{\ell=1}^n\alpha_{\ell}=n$ and each $\alpha_{\ell}=O(1)$.

\subsection{Expansion of the solution}
The solution to \eqref{eq:1D_analysisequation} can be expanded as
\begin{equation}\label{eq:expansion}
\psi(x,\mu)=\sum_{p=0}^{\infty} \lambda^{p} \psi^{(p)}(x,\mu),
\end{equation}
where $\psi^{(0)}$ satisfies
\begin{equation}\label{eq:1Disop0}
    \mu \partial_{x}\psi^{(0)}(x,\mu)+\sigma_{T}(x)\psi^{(0)}(x,\mu)=q(x),
\end{equation}
with the boundary conditions \eqref{eq:1D_bound}. Moreover,
$\psi^{(p)}$ ($p \ge 1$) satisfies
\begin{equation}\label{eq:1Disop1}
\mu \partial_{x}\psi^{(p)}(x,\mu)+\sigma_{T}(x)\psi^{(p)}(x,\mu)=\sigma_{r}(x)\phi^{(p-1)}(x)
\end{equation}
with zero inflow boundary conditions. Here,
\begin{gather}\label{eq:phiI}
\phi^{(p-1)}(x)=\mathcal{I}(\psi^{(p-1)})(x),
\end{gather}
where $\mathcal{I}$ is the operator in \eqref{eq:regularizedsource}.
If $\lambda<1$, $\int |q/\sigma_r|^2\sigma_T\,dx<\infty$ and the boundary data are bounded, the summation on the right hand side of \eqref{eq:expansion} converges in certain sense (see Corollary \ref{cor:seriesconv}) and it is easy to verify that \eqref{eq:expansion} satisfies the original equation \eqref{eq:1D_analysisequation} \cite[Chapter 3, Section 2]{1981Perturbation}. In order to prove the convergence order, we rewrite the expansion in \eqref{eq:expansion} into an operator form. Two operators: the transport operator $\cA_{\mu}$ and the iteration operator $\cT$ will first be introduced.

\paragraph{The transport operator $\cA_{\mu}$} Solving \eqref{eq:1Disop0} yields 
\begin{subequations}\label{eq:psi0}
\begin{align}
 \mu>0, \quad & \psi^{(0)}(x,\mu)=\frac{1}{\mu} \int_{x_{L}}^{x} e^{-\frac{1}{\mu} \int_{y}^{x} \sigma_{T}(z) d z} q(y) d y+ e^{-\frac{1}{\mu} \int_{x_{L}}^{x} \sigma_{T}(y) d y} \psi_L\left( \mu\right),	\\
 \mu<0, \quad & \psi^{(0)}(x,\mu)=-\frac{1}{\mu} \int_{x}^{x_{R}} e^{\frac{1}{\mu} \int_{x}^{y} \sigma_{T}(z) d z} q(y) d y+ e^{\frac{1}{\mu}\int_{x}^{x_{R}} \sigma_{T}(y) d y}\psi_R\left(\mu\right).
\end{align}
\end{subequations}
Similarity, the solution to \eqref{eq:1Disop1} is
\begin{equation}\label{eq:sol_part1}
    \psi^{(p)}(x,\mu)=
\begin{cases}
\frac{1}{\mu} \int_{x_{L}}^{x} e^{-\frac{1}{\mu} \int_{y}^{x} \sigma_{T}(z) d z} \sigma_r(y)\phi^{(p-1)}(y) d y, & \mbox{for $\mu>0$},\\
-\frac{1}{\mu} \int_{x}^{x_{R}} e^{\frac{1}{\mu} \int_{x}^{y} \sigma_{T}(z) d z} \sigma_r(y)\phi^{(p-1)}(y) d y, & \mbox{for $\mu<0$}.
\end{cases}
\end{equation}
It would be convenient to denote the solution operator in \eqref{eq:sol_part1} by $\cA_{\mu}$ such that
\begin{equation}\label{eq:A_mu}
\psi^{(p)}(\cdot,\mu):=\cA_{\mu}(\phi^{(p-1)}).
\end{equation}
Then, the solution in \eqref{eq:psi0} can be rewritten as
\begin{equation}\label{eq:psi0b}
    \psi^{(0)}(x,\mu)=\cA_{\mu}(q/\sigma_{r})(x)+b_{\mu}(x),
\end{equation}
where
\[
b_{\mu}(x)=
\begin{cases}
B_{\mu}(x)\psi_L(\mu), & \mu>0,\\
B_{\mu}(x)\psi_R(\mu), & \mu<0,
\end{cases} \quad \text{with} \quad
B_{\mu}(x)=
\begin{cases}
e^{-\frac{1}{\mu} \int_{x_{L}}^{x} \sigma_{T}(y) d y}, & \mu >0,\\
e^{\frac{1}{\mu}\int_{x}^{x_{R}} \sigma_{T}(y) d y}, & \mu<0.
\end{cases}
\]

In the subsequent part, we consider the space $L^2(\Omega; \sigma_T)$ with inner product
\begin{equation}\label{eq:L2norm}
\langle f, g\rangle:=\int_{x_L}^{x_R}f g \sigma_T \,dx.
\end{equation}
Then, $\cA_{\mu}$ can be viewed as the integral operator on $L^2(\Omega; \sigma_T)$ such that
\[
\cA_{\mu}(\phi)(x)
=\int_{x_L}^{x_R}k_{\mu}(x, y)\phi(y)\sigma_T(y)dy,
\]
with kernel 
\begin{gather*}
k_{\mu}(x, y)=
\begin{cases}
\frac{1}{\mu}\mathbb {I}_{(y\leq x)} e^{-\frac{1}{\mu} \int_{y}^{x} \sigma_{T}(z) d z}\frac{\sigma_r(y)}{\sigma_T(y)}, & \mu >0,\\
-\frac{1}{\mu}\mathbb {I}_{(y\ge x)} e^{\frac{1}{\mu} \int_{x}^{y} \sigma_{T}(z) d z}\frac{\sigma_r(y)}{\sigma_T(y)}, & \mu<0.
\end{cases}
\end{gather*}

\paragraph{The iteration operator $\cT$}
We then introduce the operator $\cT$:
\begin{equation}\label{eq:T}
\cT=\fint_S \cA_{\mu}d\mu,
\end{equation}
which satisfies
\begin{equation*}
\phi^{(p)}=\fint_{S} \psi^{(p)}(\cdot,\mu) d \mu =\fint_{S} \cA_{\mu}(\phi^{(p-1)}) d \mu=\cT(\phi^{(p-1)}),\quad p\ge1,
\end{equation*}
and
\begin{equation*}
\phi^{(0)}(x)=\cT(q/\sigma_{r})(x)+\fint_{S}b_{\mu}(x) d \mu.
\end{equation*}

\paragraph{Expansion in operator form}
Therefore, the average density is given by
\begin{equation}\label{eq:phi}
\phi(x)=\sum_{p=0}^{\infty}\lambda ^{p}\phi^{(p)}(x)=\sum_{p=0}^{\infty}\lambda^p \left(\cT^{p+1}(q/\sigma_r)(x)+\cT^p\fint_{S}b_{\mu}(x)\,d\mu \right).
\end{equation}


Similar as for \eqref{eq:phi}, the density $\phi^{\xi}$ for the $\xi$th sample of ROM can be written as
\begin{equation}\label{eq:phi_xi}
\phi^{\xi}(x)
=\sum_{p=0}^{\infty}\lambda^p\left((\cT^{\xi})^{p+1}(q/\sigma_r)(x)+(\cT^{\xi})^{p}\left(\frac{1}{n}\sum_{\ell} \alpha_{\ell} b_{\mu^\xi_{\ell}}(x) \right)\right),
\end{equation}
where the iteration operator $\cT^\xi$ becomes
\begin{equation}\label{eq:T_xi}
\cT^{\xi}=\frac{1}{n}\sum_{\ell}\alpha_{\ell}  \cA_{\mu^\xi_{\ell}}.
\end{equation}
Our goal is to estimate the difference between $\phi$ and $\phi^{\xi}$, using the expansions listed above.

\section{Main result and the proof}\label{sec:result}
As demonstrated in our numerical paper \cite{li2024randomordinatemethodmitigating}, when the solution regularity is low in velocity space, the convergence orders of given quadratures are also low.  However,  the convergence order of the average of multiple samples is higher even if the solution regularity in velocity space is low. Therefore, if we take more samples of $\mathbb{V}^\xi$, run the system \eqref{eq:ROMslab} multiple times in parallel, and then take the expectation $\mathbb{E}\phi^{\xi}$, the solution accuracy can be improved. We aim to prove this fact for ROM. The main approach is to take advantage of the expansion of solutions in the previous section. A technique difficulty is to establish a concentration inequality for the operator-valued random vectors.

For a given quadrature $\mathbb{V}^\xi$, one can measure the numerical errors by the difference between $\phi^\xi(\b{z})=\sum_{\b{u}_\ell\in \mathbb{V}^\xi}\omega_\ell\psi_\ell(\b{z})$ and the reference density $\phi(\b{z})$.  We consider the following two quantities.
The first quantity is the bias 
\begin{gather}
\mathcal{B}:=\left\|\mathbb{E} \phi^{\xi}-\phi\right\|,
\end{gather}
 which gives the distance between the expected value of $\phi^{\xi}$ and the reference solution. Here the norm is a chosen suitable norm for the functions $\phi: [x_L, x_R]\to \R$. This quantity characterizes the systematic error of the method, which will not vanish if we simply increase the number of samples. We will see in the subsequent part that the quantity is not zero, because $\mathbb{E}(\mathcal{T}^{\xi}) = \mathcal{T}$ but $\mathbb{E}[(\mathcal{T}^{\xi})^p]$ is not necessarily equal to $\mathcal{T}^p$. The second quantity is the expected single-run error 
\begin{gather}
\mathcal{E}:=\E\|\phi^{\xi}-\phi\|.
\end{gather}
Note that this expected single-run error could be used to control the variance of the method by H\"older's inequality. 

We will show that the expected error is of order $3/2$ while the bias is of order $3$, which indicates that a single typical run of ROM can give a $3/2$ order of convergence and the expectation of multiple runs gives a $3$rd order convergence. This then justifies the numerical findings in \cite{li2024randomordinatemethodmitigating} (for the original RTE model).

The main convergence result for ROM is the following.
\begin{theorem}[\bf{main result}]\label{thm:errortransport}
Consider the regularized model for \eqref{eq:1D_analysisequation} and suppose that the rescaled weights $\alpha_{\ell}$ are uniformly bounded. Then, there exists $n_0>0$ such that for $n>n_0$, the expected single run error satisfies
\begin{equation}\label{Expphi1}
\E \|\phi^{\xi}-\phi\|\le C n^{-3/2}(\log  n)^{1/2},
\end{equation}
and the bias satisfies
\begin{equation}\label{Expphi2}
\|\E\phi^{\xi}-\phi\|\le C\lambda  n^{-3}\log n.
\end{equation}
Here, the norm used is the $L^2(\sigma_T)$ norm in \eqref{eq:L2norm}.
\end{theorem}

 \begin{proof}[Proof of Theorem \ref{thm:errortransport}]
Let
$
\cT_{\ell}^{\xi}=\alpha_{\ell}\cA_{\mu^\xi_{\ell}},
$
and define
\begin{equation}\label{eq:expT}
\cT_{\ell}=\alpha_{\ell}\fint_{S_{\ell}} \cA_{\mu}d\mu=n \omega_{\ell} \frac{1}{|S_{\ell}|}\int _{S_{\ell}}\cA_{\mu}d \mu=\frac{n}{|S|}\int_{S_{\ell}}\cA_{\mu} d\mu.
\end{equation}
Therefore, from the definition of  $\cT_{\ell}$ and $\cT_{\ell}^{\xi}$,
\begin{align}\label{eq:deltaT_xi}
\delta \cT^{\xi}&:=\cT^{\xi}-\cT=\frac{1}{n}\sum_{\ell=1}^n(\cT^\xi_\ell-\cT_\ell)
\nonumber\\&=\frac{1}{m}\sum_{\ell=1}^m\frac{1}{2}(\cT_{\ell}^{\xi}-\cT_{\ell}+\cT_{n+1-\ell}^{\xi}-\cT_{n+1-\ell})
=: \frac{1}{m}\sum_{\ell=1}^m \delta \cT_{\ell}^{\xi}.
\end{align}
Since in ROM, $\mu_\ell^\xi=\mu_{n+1-\ell}^\xi$, then $\delta \cT_{\ell}^{\xi}$
and $\delta \cT_{n+1-\ell}^{\xi}$ are not independent. This is why we put $\ell$ and $n+1-\ell$ together.

According to \eqref{eq:expT}, since $\mathbb{E}T_{\ell}^{\xi}=T_{\ell}$ ($\ell=1,\cdots,n$), one has 
\begin{equation}\label{eq:EdeltaT}
\E \delta \cT_{\ell}^{\xi}=0, \quad \mbox{for $\ell=1,\cdots,m$, }\qquad \E \delta \cT^{\xi}=0. 
\end{equation}

Comparing \eqref{eq:phi} and \eqref{eq:phi_xi}, we denote
\begin{gather*}
b(x)=\fint_S b_{\mu}(x)d\mu,\quad  \delta b(x):=\frac{1}{n}\sum_{\ell}\alpha_{\ell}b_{\mu_{\ell}}(x)-b(x),
\end{gather*}
and 
\begin{equation}\label{eq:Edeltab}
\mathbb{E}\delta b=\frac{1}{n}\sum_{\ell}\alpha_{\ell}\mathbb{E}b_{\mu_{\ell}}(x)-\fint_{S}b_{\mu}(x) d\mu=\frac{1}{|S|}\sum_{\ell} \int_{S_{\ell}} b_{\mu_{\ell}}(x)d \mu_{\ell}-\fint_{S} b_{\mu}(x)d\mu=0.
\end{equation}
Here, $b_{\mu_{\ell}}$ and $b_{\mu_{n+1-\ell}}$ are not independent either. In the estimate of $\E \|\delta b\|$, one needs to put them together as well. 

Below, the norm for functions will be $L^2(\sigma_T)$ norm and the norm for the operators will be the operator norm from $L^2(\sigma_T)$ to $L^2(\sigma_T)$.
By using $(\cT)^p=(\cT^\xi-\cT+\cT)^p=(\delta\cT^\xi+\cT)^p$,  \eqref{eq:phi} and \eqref{eq:phi_xi}  the expected single run error is then controlled as below
\begin{equation*}
\begin{split}
\mathbb{E}\|\phi^{\xi}-\phi\| 
&\le  \sum_{p=0}^{\infty}\lambda^p \sum_{k=1}^{p+1} {p+1 \choose k}\cdot
\|\cT\|^{p+1-k}\cdot \|q/\sigma_r\| \cdot\mathbb{E}\|\delta \cT^{\xi}\|^k  \\
& +\sum_{p=1}^{\infty}\lambda^p \sum_{k=1}^{p} {p \choose k}
\cdot \|\cT\|^{p-k} \cdot \|b(x)\|\cdot \mathbb{E}\|\delta \cT^{\xi}\|^k\\
&+ \sum_{p=0}^{\infty}\lambda^p\sum_{k=0}^{p} {p \choose k}
\cdot \|\cT\|^{p-k} \cdot \mathbb{E}\| \delta b(x)\|\cdot \mathbb{E}\|\delta \cT^{\xi}\|^k\\
&=: E_1+E_2+E_3.
\end{split}
\end{equation*}

The bias can be controlled similarly. 
\begin{equation*}
\begin{aligned}
 \|\mathbb{E}\phi^{\xi}-\phi\| 
&\le  \sum_{p=1}^{\infty}\lambda^p \sum_{k=2}^{p+1} {p+1 \choose k} \cdot
\|\cT\|^{p+1-k} \cdot \|q/\sigma_r\| \cdot \mathbb{E}\|\delta \cT^{\xi}\|^k  \\
& +\sum_{p=2}^{\infty}\lambda^p \sum_{k=2}^{p} {p \choose k} \cdot
\|\cT\|^{p-k}  \cdot \|b(x)\| \cdot \mathbb{E}\|\delta \cT^{\xi}\|^k\\
&+\sum_{p=1}^{\infty}\lambda^p\sum_{k=1}^{p} {p \choose k}
\cdot \|\cT\|^{p-k} \cdot \mathbb{E}\| \delta b(x)\|\cdot \mathbb{E}\|\delta \cT^{\xi}\|^k\\
&=: B_1+B_2+B_3.
\end{aligned}
\end{equation*}
Here, the difference of bias from the expected single run error is that the terms involving a single $\delta\cT^{\xi}$ or $\delta b$ vanishes under expectation. For $B_1$ and $B_2$, the summation index $p$ starts respectively from 1 and 2, and the inner index $k$ starts from $k=2$. This is because, from \eqref{eq:EdeltaT}, $\mathbb{E}(\cT+\delta \cT^{\xi})=\cT$. For $B_3$, the index $p$ starts from $1$ and $k$ from $1$, because $\E \delta b=0$  by \eqref{eq:Edeltab}.

Therefore, we need to control $\|\cT\|$, $\|\cT^{\xi}\|$, $\mathbb{E}\| \delta b(x)\|$ and $\E\|\delta \cT^{\xi}\|^p$ for $p>0$ to estimate the error and bias. We can establish the following estimates.
\begin{itemize}
    \item \textit{EST I.} Lemma \ref{lmm:nonexpan} 
    shows that $\|\cT\| \le 1$, $\|\cT^{\xi}\| \le 1$ and $\sup_{\xi} \|\delta \cT^{\xi}\|\le Cn^{-1} $. 
    \item \textit{EST II.} Corollary \ref{Corollary:EdeltaT^2}
    tells that $\E(\|\delta \cT^{\xi}\|^2)  \le C (|\log n|+1) n^{-3}$.
    \item \textit{EST III.} Lemma \ref{lmm:boundaryestimate} 
    proves that $\E\|\delta b\|\le C\sqrt{n^{-3}|\log n|}$. 
\end{itemize}
The estimate for $\E(\|\delta \cT^{\xi}\|^2)$ is the most difficult one as it involves the concentration of norms for random operators in Hilbert spaces. It is established by a type of Rosenthal inequality. Other estimates are relatively straightforward. See the detailed proof in the next subsection.

Using \textit{EST I}, one then has
\begin{equation*}
    E_1\le  \|q/\sigma_r\| \left(\mathbb{E}\|\delta \cT^{\xi}\|+ \sum_{p=2}^{\infty}\lambda^{p-1} \sum_{k=1}^{p} {p \choose k}  \mathbb{E}\|\delta \cT^{\xi}\|^k \right),
\end{equation*}
and 
\begin{multline*}
\sum_{p=2}^{\infty}\lambda^{p-1} 
\sum_{k=1}^{p} {p \choose k}
\E \|\delta \cT^{\xi}\|^k
\le \sum_{p=2}^{\infty}\lambda^{p-1} 
\Bigg(p\E \|\delta \cT^{\xi}\|\\
+\frac{p(p-1)}{2}\E \|\delta \cT^{\xi}\|^2+ \sum_{k=3}^{p} {p \choose k}
\sup_{\xi}\|\delta \cT^{\xi}\|^{k-2}\E\|\delta \cT^{\xi}\|^2 \Bigg).
\end{multline*}
Since the series $\sum_{p=2}^{\infty} \lambda^{p-1}p$ and $\sum_{p=2}^{\infty}\lambda^{p-1}p^2$ converges and $\sup_{\xi} \|\delta \cT^{\xi}\|\le Cn^{-1} $, one then concludes that
\begin{gather*}
E_1\lesssim \E \|\delta \cT^{\xi}\|
+\left(1+\sum_{p=3}^{\infty}\lambda^{p-1}\sum_{k=3}^{p}{p \choose k} (C/n)^{k-2}\right)\E\|\delta \cT^{\xi}\|^2,
\end{gather*}
Since ${p \choose k}=\frac{p(p-1)}{k(k-1)} {p-2 \choose k-2}<p(p-1) {p-2 \choose k-2}$, one has $\sum_{k=3}^{p}{p \choose k} (C/n)^{k-2} < p(p-1)(1+\frac{C}{n})^{p-2}$ by \textit{EST I}. When $n$ is large enough, $\lambda(1+\frac{C}{n})<1$, then $\sum_{p=3}^{\infty}p(p-1)\left(\lambda(1+\frac{C}{n})\right)^p$ converges and the series in the front of $\E\|\delta \cT^{\xi}\|^2$ is controlled by a constant independent of $n$.
The estimation of $E_1$ then follows from the estimates of $\E\|\delta \cT ^{\xi}\|^2$ in \textit{EST II} and the H\"older inequality. 
The estimates for $E_2$  and $E_3$ are similar to $E_1$ and we omit the details. The estimates for the expected single run error then follows.

Next, we consider the bias. The estimate for the three terms are similar and we take the one for $B_1$ as the example. By the definition of $B_1$, one has
\begin{gather*}
\begin{split}
 B_1 &\le  \|q/\sigma_r\| \left(\lambda \mathbb{E}\|\delta T^{\xi}\|^2+ \sum_{p=2}^{\infty}\lambda^{p} \sum_{k=2}^{p+1} {p+1 \choose k}  \mathbb{E}\|\delta T^{\xi}\|^k \right)\\
 & \le \|q/\sigma_r\| \lambda \mathbb{E}\|\delta T^{\xi}\|^2\left(1+ \sum_{p=2}^{\infty}\lambda^{p-1} \sum_{k=2}^{p+1} {p+1 \choose k}  (C/n)^{k-2} \right)
 \end{split}
\end{gather*}
For the first inequality above,  we have used the simple bound $\|\cT\|\le 1$. For the second inequality above, we have used the fact $\mathbb{E}\|\delta T^{\xi}\|^k\le (C/n)^{k-2}\E\|\delta T^{\xi}\|^2$
due to $\sup_{\xi} \|\delta \cT^{\xi}\|\le Cn^{-1} $.  It is easy to check that the terms with $p=2$ can be controlled. The terms with $p\geq 3$,  $k=2$ are fine, due to the convergence of $\sum_{p=2}^{\infty}\lambda^{p-1} p^2$. The $k\ge 3$ terms are exactly the same as for $E_1$ . Hence, $B_1$ is bounded by a constant  that is independent of $n$ multiplying $\lambda \E\|\delta T^{\xi}\|^2$. 

 In $B_3$, when $k=1$, one may bound $\E\|\delta \cT^{\xi}\|
 \le \sqrt{\E \|\delta \cT^{\xi}\|^2}$. Besides this, the remaining estimates for $B_2$ and $B_3$ are the same as for $B_1$. We omit the details. In summary, in order to get \eqref{Expphi1} and \eqref{Expphi2}, one only needs to prove \textit{EST I, II and III}.
\end{proof}

\section{Some important estimates}
\label{sec:usefulestimates}

In this section, we provide the details for the estimates of  $\cT,\cT^{\xi},\delta\cT^{\xi}$ and $\delta b$ in \textit{EST I, II and III}.
In particular, we will establish a concentration inequality of Rosenthal inequality for operators.

Recall the inner product in \eqref{eq:L2norm} and that the operator $\cA_{\mu}(\cdot): L^2(\Omega; \sigma_T)\to L^2(\Omega; \sigma_T)$ maps $\phi(x)$ to $\psi(\cdot,\mu)$ by solving
 \begin{equation}\label{eq:mapA}
 \begin{aligned}
   & \mu \partial_{x} \psi+\sigma_{T} \psi =\sigma_{r}\phi, \quad  \mu \in S_\delta=[-1,-\delta)\cup(\delta,1],\\
   & \psi(x_L,\mu)=0,\quad \mu>0; \quad  \psi(x_R,\mu)=0,\quad \mu<0.
 \end{aligned}
 \end{equation}
where $\delta $ is a small positive constant in regularized system.

\begin{lemma}\label{lmm:nonexpan}
The operator $\cA_{\mu}(\cdot)$ satisfies that
\begin{gather}\label{eq:boundsA}
\|\cA_{\mu}(\phi) \|_{L^2(\Omega;\sigma_T)} \le \|\phi \|_{L^2(\Omega;\sigma_T)}
\Rightarrow \|\cA_{\mu}\|\le 1, \quad \forall \mu\in [-1, 1],
\end{gather}
and for $\mu\in (-1, 0)\cup (0, 1)$, one has
\begin{gather}\label{eq:lipestimate}
\|\partial_{\mu}\cA_{\mu}\|\le \frac{1}{|\mu|}\left(1+\left\|\frac{\sigma_T}{\sigma_r}\right\|_{\infty}\right).
\end{gather}

Consequently, $\|\cT\| \le 1$, $\|\cT^{\xi}\| \le 1$, and it holds for the regularized system that
\begin{equation}\label{eq:sup_deltaT_{ell}}
\|\delta \cT_{\ell}^{\xi}\|\le \frac{C}{\delta}n^{-1}.
\end{equation}
\end{lemma}
\begin{proof}
We only consider $\mu> 0$ here, as the case for $\mu<0$ is similar. Multiplying $\psi$ on both sides of \eqref{eq:mapA} and integrating, one has
\begin{gather*}
\begin{split}
\mu \frac{1}{2}|\psi(x_R, \mu)|^2+\int_{x_L}^{x_R} \sigma_T |\psi|^2\,dx
&=\int_{x_L}^{x_R}\sigma_r\phi\psi\,dx 
 \le \int_{x_L}^{x_R}\sigma_T|\phi||\psi|dx \\
 &\le \left(\int_{x_L}^{x_R}\sigma_T|\phi|^2\,dx\right)^{1/2}
\left(\int_{x_L}^{x_R}\sigma_T |\psi|^2dx \right)^{1/2}.
\end{split}
\end{gather*}
Hence,
\[
\left(\int_{x_L}^{x_R}\sigma_T |\psi|^2dx \right)^{1/2}  \le \left(\int_{x_L}^{x_R}\sigma_T|\phi|^2\,dx\right)^{1/2},
\]
which implies that $\|\cA_{\mu}\|\le 1$.

Taking the derivative of both sides of \eqref{eq:mapA} with respect to $\mu$
\[
\partial_x\psi+\mu\partial_x(\partial_{\mu}\psi)
+\sigma_T\partial_{\mu}\psi=0, 
\quad \partial_{\mu}\psi(x_L, \mu)=0.
\]
Hence, from \eqref{eq:mapA}, one has
\[
\partial_{\mu}\psi=\cA_{\mu}(-\partial_x\psi/\sigma_r)
=-\cA_{\mu}\left(\mu^{-1}(\phi-\frac{\sigma_T}{\sigma_r}\psi)\right).
\]
By \eqref{eq:boundsA}, one then has
\[
\|\partial_{\mu}\psi\|
\le \mu^{-1}\|\phi-\frac{\sigma_T}{\sigma_r}\cA_{\mu}(\phi) \|\le \mu^{-1}(\|\phi\|+\|\frac{\sigma_T}{\sigma_r}\|_{\infty}\|\phi\|).
\]
The inequality \eqref{eq:lipestimate} holds. 

Since $\cT^{\xi}$ and  $\cT$ are the averages of $\cA_{\mu}$, $\|\cT^{\xi}\|\leq 1$ and  $\|\cT\|\leq 1$ are straightforward.
For the truncated system, one has
$$\delta \cT^{\xi}_{\ell}=\alpha_{\ell}|S_{\ell}|^{-1}\int_{S_{\ell}} \cA_{\mu}-\cA_{\mu_{\ell}} d\mu,
$$
then from the definitions of  $\omega_\ell$ and $\alpha_\ell$ in  \eqref{omegal} and \eqref{alphal}, 
\[
\|\delta \cT_{\ell}^{\xi}\|\le \alpha_{\ell}\sup_{\mu\in S_{\ell} \cup S_{\ell+m}}\|\partial_{\mu}\cA_{\mu}\| |S_{\ell}|
=n^{-1}|S|\alpha_{\ell}^2\sup_{\mu\in S_{\ell}\cup S_{\ell+m}}\|\partial_{\mu}\cA_{\mu}\|.
\]
The claim then follows.
\end{proof}

\begin{corollary}\label{cor:seriesconv}
 Suppose $q/\sigma_r\in L^2(\cdot; \sigma_T)$ and the boundary data are bounded. Then, 
 \begin{gather}
\sup_{\mu, p} \|\psi^{(p)}(\cdot,\mu)\|<\infty,
 \end{gather}
and the series on the right hand side of \eqref{eq:expansion} converges in $L^2(\cdot; \sigma_T)$ for each $\mu$. Moreover, the limit is the solution to \eqref{eq:1D_analysisequation}
\end{corollary}

\begin{proof}
Since $\|\cA_{\mu}\|\le 1$,  by \eqref{eq:psi0b}, it is easy to see that
\[
\sup_{\mu}\|\psi^{(0)}(\cdot, \mu)\|<\infty.
\]
By \eqref{eq:phiI} and \eqref{eq:A_mu}, the bound for $\sup_{\mu, p} \|\psi^{(p)}(\cdot,\mu)\|$ then follows.
Since $\lambda<1$, the right hand side of \eqref{eq:expansion} converges in $L^2(\cdot; \sigma_T)$ strongly and let the limit be
$\bar{\psi}(\cdot, \mu)$.

Then, it is easy to see that the following holds in $L^2(\Omega; \sigma_T)$
\[
\bar{\psi}(\cdot, \mu)=\cA_{\mu}(\lambda \mathcal{I}(\bar{\psi})(\cdot)+q/\sigma_r)+b_{\mu}(\cdot).
\]
This indicates that $\bar{\psi}$ is the solution in $L^2(\Omega; \sigma_T)$ and agrees with the solution a.e. so that it is identified with the solution.
\end{proof}

Next, our main goal is to estimate 
$\E\|\delta \cT^{\xi}\|^2$ where $\delta \cT^{\xi}$ is given in \eqref{eq:deltaT_xi}.
If the independent random variables $\delta \cT_{\ell}^{\xi}$ take values in a Hilbert space, some classical concentration inequalities like the Rosenthal inequality \cite{1970On} can be used to achieve this. 
However, here the variables are operators over Hilbert spaces so they are in a Banach algebra. One cannot apply the classical Rosenthal inequality directly. Our goal in this section is then to establish a Rosenthal type inequality for these operator-valued random variables.

First, we observe
\[ 
\E\|\delta \cT^{\xi}\|^2 \le \E\big|\|\delta \cT^{\xi}\|-\E\|\delta \cT^{\xi}\| \big|^2+(\E\|\delta \cT^{\xi}\|)^2,
\]
where the second term on the right-hand side can be estimated by \eqref{eq:sup_deltaT_{ell}} directly, and the first term on can be estimated by the following fact.
\begin{lemma}\label{lmm:deviatesfrommean}
For any separable Banach space $B$ and any finite sequence of independent $B$-valued random vectors $X_j$, $1\le j\le m$ with $\E\|X_j\|^2<\infty$. Let $S_{m}=\sum_{j=1}^{m} X_{j}$. Then,
$$
\mathbb{E}\big|\left\|S_m\right\|-\mathbb{E}\left\|S_m\right\|\big|^2 \leq 4 \sum_{j=1}^m \mathbb{E}\left\|X_j\right\|^2.
$$
\end{lemma}
This result is a special case of 
 \cite[Theorem 2.1]{BvaluedInequalities}, where the general $L^p$ case is considered. The proof is a consequence of the Birkholder inequality for martingales. Here, we sketch the proof for the special case here.
 Consider the filtration $\{\cF_{j}, 0\le j\le m\}$ where 
 \[
 \cF_j=\sigma(X_{\ell}: 1\le \ell \le j), 1\le j\le m,
 \]
 and $\cF_0=\{\Omega, \emptyset\}$.
 Define $Y_j=\E(\|S_m\||\cF_j)-\E(\|S_m\||\cF_{j-1})$, where $\E(\|S_m\||\cF_0):=\E(\|S_m\|)$. Then, one has
 \begin{gather*}
\mathbb{E}\big|\left\|S_m\right\|-\mathbb{E}\left\|S_m\right\|\big|^2
=\E|\sum_{j=1}^m Y_j|^2=\sum_{j=1}^m \E(|Y_j|^2).
 \end{gather*}
 By the definition, one has
 \begin{multline*}
|Y_j|=\Bigg| \E_{X_1,\cdots, X_{j-1}, X_j'}\|X_1+\cdots+X_{j-1}+X_j'+X_{j+1}+\cdots\| \\
-\E_{X_1,\cdots, X_{j-1}}\|X_1+\cdots+X_{j-1}+X_j+X_{j+1}+\cdots\|\Bigg|
\le \E_{X_j'}|X_j'-X_j|.
 \end{multline*}
 Here $X_j'$ is an independent copy of $X_j$. This then gives
 that $\E|Y_j|^2\le 4\E(\|X_j\|^2)$.


Hence, the problem is reduced to estimation of $\E\|\delta \cT^{\xi}\|$. We will mainly make use of the approach in \cite{tropp2016expected}. The analysis in \cite{tropp2016expected} is for matrices and the final result relies on the dimension of the space. In our case, the operator is in an infinite dimensional space. We find that the dependence on the dimension is due to the trace of the operators. Fortunately, for our case, the operator is compact and we could possibly bound the trace.

Firstly, we introduce the Mercer's theorem about trace class mentioned in \cite[Sec. 30.5, Theorem 11]{2007Functional}.
\begin{lemma}\label{lmm:traceclass}
Consider an integral operator $\b{K}$ of the form 
\begin{equation*}
  (\b{K}u)(s)=\int_{I}K(s,t)u(t) w(t)d t
\end{equation*}
acting on $L^2(I; w)$, where $K$ is a real-valued symmetric, continuous function of $(s, t)$ and $w(t)$ is a continuous positive weight. Then, the operator $\b{K}$ is positive in the usual sense:
\begin{equation*}
(\b{K}u,u) \ge 0,\quad  \text{for all $u$ in $L^2(I; w)$},
\end{equation*}
 and is of trace class with the trace equal to the integral of its kernel along the diagonal:
 \begin{equation*}
\tr(\b{K})=\int_I K(s,s)w(s)ds.
 \end{equation*}
\end{lemma}
The result in \cite[Sec. 30.5, Theorem 11]{2007Functional} is about the uniform weight $w=1$. It is not hard to check that the proof holds for general continuous positive weight on $I$.  Based on the above lemma, we naturally deduce the following proposition.
\begin{proposition}\label{prop:compactandtrace}
For any $\mu\neq 0$, $\cA_{\mu}$ is a compact operator. $\cA_{\mu}^*$ is the adjoint operator of $\cA_{\mu}$, then $\cA_{\mu}^* \cA_{\mu}$ and $\cA_{\mu} \cA_{\mu}^*$ are in the trace class. Moreover,
\begin{equation*}
\tr \cA_{\mu}^* \cA_{\mu}=\tr \cA_{\mu} \cA_{\mu}^*
\le \frac{1}{|\mu|}|x_R-x_L|\|\sigma_T\|_{\infty}^2 \|\sigma_T^{-1}\|_{\infty}.
\end{equation*}
\end{proposition}

\begin{proof}
According to \cite{2007Functional}, an integral operator $L^2(I; w) \to L^2(I; w)$ with a square integrable kernel is compact. The first claim follows directly by 
the solution in expanded form.

 We only consider $\mu>0$ and the proof for $\mu<0$ is similar. By the definition of adjoint operators,
 $\langle \cA_{\mu} g,h \rangle=\langle g, \cA_{\mu}^* h\rangle$, one finds that
 \begin{equation*}
 (\cA_{\mu}^*\varphi)(x)=\int_{x_L}^{x_R}\tilde{k}_{\mu}(x, y)\varphi(y)\sigma_T(y)dy, \quad \tilde{k}_{\mu}(x, y)
 =k_{\mu}(y, x).
 \end{equation*}
Then, one has
\[
\cA_{\mu}\cA_{\mu}^*\varphi =\int_{x_L}^{x_R}K_1(x, y)\varphi(y)
\sigma_T(y)dy, \quad K_1(x, y)=\int_{x_L}^{x_R}k_{\mu}(x, z)\tilde{k}_{\mu}(z, y)\sigma_T(z)dz.
\]
Similarly,
\[
\cA_{\mu}^*\cA_{\mu}\varphi =\int_{x_L}^{x_R}K_2(x, y)\varphi(y)
\sigma_T(y)dy, \quad K_2(x, y)=\int_{x_L}^{x_R}\tilde{k}_{\mu}(x, z)k_{\mu}(z, y)\sigma_T(z)dz.
\]
Hence,
\begin{subequations}\label{eq:K}
\begin{align}
K_1(x, y)=\int_{x_L}^{\min(x, y)}\frac{1}{\mu^2}\frac{\sigma_r^2(z)}{\sigma_T(z)}\exp\left(-\frac{1}{\mu}(\int_z^x\sigma_T(w)dw+\int_z^y\sigma_T(w)dw)\right)&dz\\
K_2(x, y)=\int_{\max(x,y)}^{x_R}\frac{1}{\mu^2}\frac{\sigma_r(x)\sigma_r(y)}{\sigma_T(x)\sigma_T(y)}\exp\left(-\frac{1}{\mu}(\int_x^z\sigma_T(w)dw+\int_y^z\sigma_T(w)dw)\right)\sigma_T(z)&dz.
\end{align}
\end{subequations}
According to these two formulas, it is clear that both kernels are continuous in $(x, y)$.

Repeating the Lemma \ref{lmm:traceclass},  
it is easy to find that both operators are in trace class and
\begin{equation*}
\tr(\cA_{\mu}\cA_{\mu}^*)
=\int_{x_L}^{x_R}K_1(x, x)\sigma_T(x)d x, 
\quad \tr(\cA_{\mu}^* \cA_{\mu})
=\int_{x_L}^{x_R}K_2(x, x)\sigma_T(x)d x.
\end{equation*}
These two traces are in fact equal by \eqref{eq:K} and Fubini's theorem, and are equal to
\[
\int_{x_L}^{x_R}\int_{x_L}^{x_R} |k_{\mu}(x, y)|^2 \sigma_T(x)\sigma_T(y)dxdy.
\] 
Then, one finds
\begin{equation}
\begin{split}
&\tr(\cA_{\mu}\cA_{\mu}^*)=\tr(\cA_{\mu}^*\cA_{\mu})\\
&\le \frac{1}{\mu^2}\|\sigma_T\|_{\infty}^2\int_{x_L}^{x_R}\int_{x_L}^{x_R}\mathbb {I}_{z\le x}
\exp\left(-\frac{2}{\mu}\frac{1}{\|\sigma_{T}^{-1}\|_{\infty}}(x-z)\right)dxdz\\
&\le \frac{1}{\mu} |x_R-x_L|\|\sigma_T\|_{\infty}^2\|\sigma_T^{-1}\|_{\infty}.
\end{split}
\end{equation}
The case for $\mu<0$ is similar, omitted.
\end{proof}

Next, we adopt the argument in \cite{tropp2016expected} to estimate $\E\|\delta \cT^{\xi}\|$ in our case.

First, define the symmetrization of the operators. 
Because $\delta \cT^{\xi}$ is not a self-adjoint operator. We  let $ \delta \cH$ be the symmetrization operator of $\delta \cT^{\xi}$, given by 
\begin{gather}
\delta \cH=\frac{1}{m}\sum_{\ell=1}^m \delta \cH_{\ell},
\end{gather}
where 
\begin{gather}
\delta \cH_{\ell}:=\left[\begin{array}{cc}
0 &  \delta \cT_{\ell}^{\xi} \\
( \delta \cT_{\ell}^{\xi})^* & 0
\end{array}\right].
\end{gather}
Clearly, each $\delta \cH_{\ell}$ are operators in $(L^2(\sigma_T))^{\otimes 2}\to (L^2(\sigma_T))^{\otimes 2}$.  The symmetrization has a good property that is
\begin{gather}\label{eq:normequalsym}
\|\delta \cT_{\ell}^{\xi}\|=\|\delta \cH_{\ell}\|,\quad
\|\delta \cT^{\xi}\|=\|\delta \cH \|.
\end{gather}
In fact, since $\delta \cH_{\ell}$ is a self-adjoint operator, and
\begin{equation*}
\begin{aligned}
    \|\delta \cH_{\ell}\|^2 & =\|\delta \cH_{\ell}^2\|=\Bigg\|\left[\begin{array}{cc}
\delta \cT_{\ell}^{\xi}(\delta \cT_{\ell}^{\xi})^* &  0 \\
 0 & (\delta \cT_{\ell}^{\xi})^*\delta \cT_{\ell}^{\xi}
\end{array}\right]\Bigg\| \\
& =\max\{\|\delta \cT_{\ell}^{\xi}(\delta \cT_{\ell}^{\xi})^*\|,\|(\delta \cT_{\ell}^{\xi})^*\delta \cT_{\ell}^{\xi}\|\}=\|\delta \cT_{\ell}^{\xi}\|^2.
\end{aligned}
\end{equation*}
The other relation $\|\delta \cT^{\xi}\|=\|\delta \cH \|$ follows from the same argument.
Hence, it reduces to estimate $\|\delta \cH\|$. 

Then, we consider the Rademacher symmetrization:
\begin{gather*}
\delta \cH^{\e}=\frac{1}{m}\sum_{{\ell}=1}^m \epsilon_{\ell} \delta \cH_{\ell},
\end{gather*}
where $\epsilon_{\ell}$ are independent Rademacher random variables (i.e., indepedent Bernoulli variables taking values in $\{-1, 1\}$). Introduction of $\e_{\ell}$ brings extra randomness to utilize the independence. 
The following is well-known (see \cite[Fact 3.1]{tropp2016expected}), which indicates that introduction of $\e_{\ell}$ would not lose control of the random variables.
\begin{lemma}\label{lmm:Rad}
The Rademacher symmetrization satisfies
\[
\frac{1}{2}\E\|\delta \cH^{\e}\| \le \E\|\delta \cH \|\le 2 \E\|\delta\cH^{\e}\|,
\]
where $\E$ means that the expectation is with respect to all randomness, including those in Rademacher variables and random ordinates. 
\end{lemma}

The following estimate gives the control by taking the expectation over $\e_{\ell}$, following the approach in \cite{tropp2016expected}.
\begin{lemma}\label{lmm:expectedmean}
It holds that
\begin{equation*}
    \mathbb{E}_{\epsilon}\|\delta\cH^{\e}\| \leq \sqrt{3+2 \left[\log\frac{m^{-1}\sum_{\ell} \tr\left(\delta \cH_{\ell}^2 \right)}{m^{-1}\sum_{\ell} \|\delta \cH_{\ell}^2\|}\right] }\left(\frac{1}{m^2}\sum_{\ell} \|\delta \cH_{\ell}^2\|\right)^{1 / 2},
\end{equation*}
where $\E_{\epsilon}$ means the expectation over the Rademacher variables.
\end{lemma}
\begin{proof}
Note that $\delta\cH^{\e}$ is self-adjoint, and $(\delta \cH^{\e})^2$ is nonnegative and compact. Hence, $\|\delta \cH^{\e}\|^{2p}
=\|(\delta \cH^{\e})^{2p}\|$ for each nonnegative integer $p$ so that
\[
\mathbb{E}_{\e}\|\delta \cH^{\e}\| \le (\mathbb{E}_{\e}\left\|(\delta\cH^{\e})^{2p}\right\|)^{1/(2p)} \leq
\left[\E_{\e} \tr\left((\delta\cH^{\e})^{2 p}\right)\right]^{1 /(2p)}.
\]
The above holds because the norm of $(\delta \cH^{\e})^2$ is the largest eigenvalue while the trace is the sum of all eigenvalues. 
For each index $\ell\in \{1,\cdots, m\}$, define the random operators
\[
\delta\cH_{+\ell}:=\frac{1}{m}\delta H_{\ell}+\frac{1}{m}\sum_{j \neq \ell} \epsilon_j \delta \cH_j \quad \text { and } \quad \delta\cH_{-\ell}:=-\frac{1}{m}\delta \cH_{\ell}+\frac{1}{m}\sum_{j \neq \ell} \epsilon_j \delta \cH_j.
\]
Denote $\hat{\e}_{\ell}=\{\e_1,\cdots, \e_{\ell-1}, \e_{\ell+1}, \cdots, \e_m \}$.  Then, it holds that
\[
\begin{aligned}
 \mathbb{E}_{\e} \operatorname{tr}\left((\delta\cH^{\e})^{2 p}\right)
 &=\frac{1}{2} \sum_{\ell} \mathbb{E}_{\hat{\e}_{\ell}} \operatorname{tr}\left(\frac{1}{m}\delta \cH_{\ell}\left(\delta\cH_{+\ell}^{2 p-1}-\delta\cH_{-\ell}^{2 p-1}\right)\right) \\
& =\frac{1}{m^2}\sum_{\ell=1}^m \sum_{j=0}^{2 p-2} \mathbb{E}_{\hat{\e}_{\ell}} \operatorname{tr}\left(\delta \cH_{\ell} \delta\cH_{+\ell}^j \delta \cH_{\ell} \delta\cH_{-\ell}^{2 p-2-j}\right)\\
&\leq \frac{1}{m^2}\sum_{\ell=1}^m \frac{2 p-1}{2} \mathbb{E}_{\hat{\e}_{\ell}} \operatorname{tr}\left[\delta \cH_{\ell}^2 \cdot\left(  \delta\cH_{+\ell}^{2 p-2}+\delta\cH_{-\ell}^{2 p-2}\right)\right] \\
&=(2 p-1)  \tr\left[\frac{1}{m^2}\left(\sum_{\ell=1}^m \delta \cH_{\ell}^2\right) \E_{\e} (\delta\cH^{\e})^{2 p-2}\right].
\end{aligned}
\]
The second line is due to the equality
\[
\delta\cH_{+\ell}^{2 p-1}-\delta\cH_{-\ell}^{2 p-1}=\sum_{q=0}^{2 p-2} \delta\cH_{+\ell}^q(\delta\cH_{+\ell}-\delta\cH_{-\ell}) \delta\cH_{-\ell}^{2 p-2-q},
\]
and $\delta\cH_{+\ell}-\delta\cH_{-\ell}=\frac{2}{m}\delta \cH_{\ell}$.
The third line is due to the following Geometric Mean-Arithmetic Mean (GM-AM) trace inequality in  \cite[Fact 2.4]{tropp2016expected} 
\begin{gather}\label{eq:gmam}
\tr(HW^q H Y^{2r-q})+\tr(HW^{2r-q}HY^q)\le \tr(H^2(W^{2r}+Y^{2r})),
\end{gather}
which can be generalized to 
self-adjoint compact operators in trace class. Here, $q, r$ are integers and $0\le q\le 2r$, and $H, W$ are two arbitrary self-adjoint compact operators in trace class. 
In fact, one can approximate the compact operators using finite rank self-adjoint operators, and the finite rank operators (essentially matrices) satisfy \eqref{eq:gmam}. Passing the limit for the finite rank approximation then verifies the inequality for self-adjoint compact operators.
The last line follows since $\delta\cH_{+\ell}^{2 p-2}+\delta\cH_{-\ell}^{2 p-2}=2\E_{\e_{\ell}}(\delta \cH^{\e})^{2p-2}$. 

Recall that (see \cite[section 30.2, Theorem 2]{2007Functional}) if $A$ is self-adjoint, nonnegative and in trace class, then $\|A\|_{\tr}=\tr(A)$ and
\begin{gather}\label{eq:traceandnorm}
\tr(AB)\le \|AB\|_{\tr} \le \|B\| \|A\|_{\tr}=\|B\|\tr(A).
\end{gather}
Applying \eqref{eq:traceandnorm} and repeating the above process, one then has
\begin{equation*}
\begin{split}
\mathbb{E}_{\e} \operatorname{tr}\left((\delta\cH^{\e})^{2 p}\right)&\le 
(2 p-1)  \left\|\frac{1}{m^2}\left(\sum_{\ell=1}^m \delta H_{\ell}^2\right)\right\| \E_{\e}\tr (\delta\cH^{\e})^{2 p-2} \\
&\le (2p-1)!!  \left\|\frac{1}{m^2}\left(\sum_{\ell=1}^m \delta H_{\ell}^2\right)\right\|^{p-1}  \E_{\e}\tr (\delta\cH^{\e})^{2}.
\end{split}
\end{equation*}
Since 
\[
\E_{\e}(\delta\cH^{\e})^{2}=\frac{1}{m^2}\sum_{\ell} \delta \cH_{\ell}^2
\]
and $(2 p-1) ! ! \leqslant\left(\frac{2 p+1}{\mathrm{e}}\right)^p$, we then arrive at
\begin{equation*}
    \mathbb{E}_{\epsilon}\|\delta\cH^{\e}\| \leq \sqrt{\frac{2 p+1}{e}}\left(\frac{1}{m^2}\sum_{\ell} \|\delta \cH_{\ell}^2\|\right)^{1 / 2-1 /(2 p)}\left(\frac{1}{m^2}\sum_{\ell} \tr\left(\delta \cH_{\ell}^2 \right)\right)^{1 /(2 p)}.
\end{equation*}
Taking 
\[
p=\left[\log\frac{\sum_{\ell} \tr\left(\delta \cH_{\ell}^2 \right)}{\sum_{\ell} \|\delta \cH_{\ell}^2\|}\right]+1
\]
gives the result.
\end{proof}

Combining Lemma \ref{lmm:deviatesfrommean} and Lemma \ref{lmm:expectedmean}, we conclude a Rosenthal type inequality for $\delta\cT^{\xi}$.
\begin{theorem}
\label{thm:ros}
For  $p=2$, $\delta \cT^{\xi}$ defined in \eqref{eq:deltaT_xi}, it holds that
\begin{gather}\label{eq:roththalp2}
\E(\|\delta \cT^{\xi}\|^2)\le 
8 \E \left[\left(\log\frac{m^{-1}\sum_{\ell} \tr\left((\delta \cT_{\ell}^{\xi})^*\delta \cT_{\ell}^{\xi} \right)}{m^{-1}\sum_{\ell} \|\delta \cT_{\ell}^{\xi}\|^2}
+2.7\right)
\frac{1}{m^2}\sum_{\ell=1}^m \|\delta \cT_{\ell}^{\xi}\|^2 \right].
\end{gather}
\end{theorem}
\begin{proof}
By Lemma \ref{lmm:Rad}, one finds that
\[
\E\|\delta \cT^{\xi}\|
=\E\|\delta\cH\|\le 2\E\|\delta \cH^{\e}\|
=2\E(\E_{\e}\|\delta \cH^{\e}\|).
\]


Then according to  Lemma \ref{lmm:deviatesfrommean} and Lemma \ref{lmm:expectedmean}, one has
\begin{equation*}
\begin{split}
\E(\|\delta\cT^{\xi}\|^2)&\le (\E\|\delta \cT^{\xi}\|)^2+4\sum_{\ell=1}^m\frac{1}{m^2}\E\|\delta \cT_{\ell}^{\xi}\|^2\\
&\le 4 \E\left[\left(4+2\log\frac{m^{-1}\sum_{\ell} \tr\left((\delta \cT_{\ell}^{\xi})^*\delta \cT_{\ell}^{\xi} \right)}{m^{-1}\sum_{\ell} \|\delta \cT_{\ell}^{\xi}\|^2}
+2\log 2\right) 
\frac{1}{m^2}\sum_{\ell} \|\delta \cT_{\ell}^{\xi}\|^2\right],
\end{split}
\end{equation*}
where we have used $[x]\le x$. Since $2\log 2<1.4$, the result follows.
\end{proof}

 We conclude the following. 
\begin{corollary}[\bf{Bounds of $\E\|\delta \cT^{\xi}\|^2$}]\label{Corollary:EdeltaT^2}
It holds for the truncated system that
\begin{gather*}
\E(\|\delta \cT^{\xi}\|^2)  \le C (\log n+1 )n^{-3},
\end{gather*}
where $C$ is a constant that is independent of $n$.
\end{corollary}

\begin{proof}
In \eqref{eq:roththalp2}, we note that
\begin{equation*}
  \E  \tr((\delta \cT_{\ell}^{\xi})^* \delta \cT_{\ell})
  =\E \tr((\bar{\cT}_{\ell}^{\xi})^*\bar{\cT}_{\ell}^{\xi})-
  \tr(\bar{\cT}_{\ell}^* \bar{\cT}_{\ell})
\leq \E \tr((\bar{\cT}_{\ell}^{\xi})^*\bar{\cT}_{\ell}^{\xi}),
\end{equation*}
where $\bar{\cT}_{\ell}=\frac{1}{2}(\cT_{\ell}+\cT_{\ell+m})$ and $\bar{\cT}_{\ell}^{\xi}$ is similarly defined. The inequality above holds because $\tr( \bar{\cT}_{\ell}^* \bar{\cT}_{\ell} )=\tr(\bar{\cT}_{\ell} \bar{\cT}_{\ell}^*)\ge 0$.
Moreover, using the simple control, 
\[
\tr(T_1T_2^*+T_2T_1^*)\le \tr(T_1 T_1^*+T_2T_2^*),
\quad \tr(T_1^*T_2+T_2^*T_1)\le \tr(T_1^* T_1+T_2^*T_2)
\]
 we conclude that by Lemma \ref{prop:compactandtrace} for the truncated system that
\[
m^{-1}\sum_{\ell} \tr (\delta \cT_{\ell}^{\xi})^*\delta \cT_{\ell}^{\xi}
\le C .
\]

Moreover, by Lemma \ref{lmm:nonexpan} and noting $n=2m$,
\[
\frac{1}{m^2}\sum_{\ell} \|\delta \cT_{\ell}^{\xi}\|^2\le C\frac{1}{n^3},
\quad |\log(m^{-1}\sum_{\ell} \|\delta \cT_{\ell}^{\xi}\|^2)|
\le C(1+\log n).
\]

The claim then follows.
\end{proof}

Next, we move to the estimation of $\|\delta b(x)\|$, which is much easier than $\delta\cT^{\xi}$ since it is in the Hilbert space $L^2(\sigma_T)$. Note that the proof here is uniform in $\delta$ so the estimate here actually applies to the original RTE model without regularization.
\begin{lemma}\label{lmm:boundaryestimate}
Suppose that $\mu\mapsto \psi_{\mu}(x_L)$ is Lipschitz on $(-1, 0)$ and $\mu\mapsto \psi_{\mu}(x_R)$ is Lipschitz on $(0, 1)$.  Then, it holds that
\[
\E\|\delta b\|_{L^2(\sigma_T)}\le C\sqrt{n^{-3}(1+\log n)},
\]
where $C$ is a constant that is independent of $n$.
\end{lemma}
\begin{proof}
By the H\"older inequality 
\[
\E\|\delta b\|\le \sqrt{\E\|\delta b\|^2}.
\]
Define $\tilde{b}_{\ell}:=b_{\mu_{\ell}}+b_{\mu_{\ell+m}}$.
Then it holds that
\begin{equation}\label{eq:norm_deltab}
\begin{aligned}
    \mathbb{E}\|\delta b\|^2 &= \mathbb{E}\Big\|\frac{1}{n}\sum_{\ell}\alpha_{\ell}b_{\mu_{\ell}}-\frac{1}{2}\int_{-1}^{1} b_{\mu}d\mu \Big \|^2
     =  \mathbb{E}\Big\|\sum_{\ell=1}^m\left(\omega_{\ell}\tilde{b}_{\ell}-\frac{1}{|S|}\int_{S_{\ell}\cup S_{\ell+m}}b_{\mu}d \mu \right)\Big\|^2\\
     &=\sum_{\ell=1}^m\mathbb{E}\Big\|\omega_{\ell}\tilde{b}_{\ell}-\frac{1}{|S|}\int_{S_{\ell}\cup S_{\ell+m}}b_{\mu}d \mu \Big\|^2\\
     &=\frac{1}{|S|^2}\sum_{\ell=1}^m \mathbb{E}\Big\|\int_{S_{\ell}} (b_{\mu_{\ell}}-b_{\mu}) d \mu
     +\int_{S_{\ell+m}} (b_{\mu_{\ell+m}}-b_{\mu}) d \mu\Big\|^2\\
    & \le \frac{2}{|S|^2}\sum_{\ell=1}^m\mathbb{E} \left(\int_{S_{\ell}} \|b_{\mu_{\ell}}-b_{\mu}\| d \mu\right)^2\le \sum_{\ell=1}^m\frac{2|S_{\ell}|}{|S|^2}\E \int_{S_{\ell}}
    \|b_{\mu_{\ell}}-b_{\mu}\|^2 d\mu.
\end{aligned}
\end{equation}
Above, the summation can be moved out of the norm because the expectation of the cross terms are zero, i.e.,
\[
\E \left(\omega_{\ell}\tilde{b}_{\ell}-\frac{1}{|S|}\int_{S_{\ell}\cup S_{\ell+m}}b_{\mu}d \mu\right)
\left(\omega_{\ell'}\tilde{b}_{\mu_{\ell'}}-\frac{1}{|S|}\int_{S_{\ell'}\cup S_{\ell'+m}}b_{\mu}d \mu\right)=0
\]
for $\ell\neq \ell'$. The last inequality is due to the H\"older inequality.

Using the explicit formula of $b_{\mu}$, we will show now for $\ell=1,\cdots, m$ and $\mu\in S_{\ell}$ that
\begin{equation}\label{eq:deltabnorm}
    \|b_{\mu_{\ell}}-b_{\mu}\|^2
\le 
\begin{cases}
C\frac{1}{|\mu|n^2} & \inf\{|\mu|: \mu\in S_{\ell} \}\ge 2n^{-1},\\
Cn^{-1} &  \text{otherwise}.\\
\end{cases}
\end{equation}
Below, we consider only $\mu>0$ ($\mu <0$ is similar).
Recall the boundary propagator for $\mu>0$:
\begin{gather*}
B_{\mu}(x)=\exp\left(-\frac{1}{\mu}\int_{x_L}^x\sigma_T(y)\,dy \right).
\end{gather*}
Clearly, for $\mu_1, \mu_2\in S_{\ell}$, $\mu_1>0, \mu_2>0$ and $\mu_1<\mu_2$, one has $\varepsilon:=\mu_2-\mu_1\le \alpha_{\ell}|S|/n$ and thus
\[
 \|b_{\mu_1}-b_{\mu_2}\|^2
 \le 2\|B_{\mu_1}(\cdot)-B_{\mu_2}(\cdot)\|^2\sup_{\mu}|\psi(x_L, \mu)|+2\alpha_{\ell}^2|S|^2\|B_{\mu_2}(\cdot)\|^2 L_{\psi}^2n^{-2},
 \]
 where $L_{\psi}$ is the Lipschitz constant of $\psi$.

 Let $M(x)=\int_{x_{L}}^{x} \sigma_{T}(y) d y$.   Note the simple fact
\begin{equation*}
B_{\mu_1}(x)-B_{\mu_2}(x)=\exp\left(-\frac{M(x)}{\mu_2}\right)\left( \exp\left(-\frac{M(x)\varepsilon}{\mu_1 \mu_2}\right)-1\right).
\end{equation*}

If $\inf\{\mu: \mu\in S_{\ell}\}<2/n$, since $\alpha_{\ell}$ is bounded, $\mu_2\le C/n$ for some constant $C>0$. Since
$|\exp\left(-\frac{M(x)\varepsilon}{\mu_1 \mu_2}\right)-1|<1$, one has
\begin{equation*}
 \left\|B_{\mu_1}(\cdot)-B_{\mu_2}(\cdot)\right\|^2
 \le \left\|\exp\left(-\frac{M(x)}{\mu_2}\right)\right\|^2 \le C n^{-1}. 
\end{equation*}

If $\inf\{\mu: \mu\in S_{\ell}\}\ge 2/n$, then $\mu_2/\mu_1= 1+\varepsilon/\mu_1$ is bounded. Using the simple bound $|\exp\left(-\frac{M(x)\varepsilon}{\mu_1 \mu_2}\right)-1|\le M(x)\varepsilon/(\mu_1\mu_2)$, one has
\begin{equation*}
\begin{split}
 \left\|B_{\mu_1}(\cdot)-B_{\mu_2}(\cdot)\right\|^2
& \le \int_{x_L}^{x_R} \frac{\varepsilon^2}{\mu_1^2\mu_2^2}M^2(x)\exp\left(-\frac{2M(x)}{\mu_2}\right)\sigma_T(x)dx \\
& \le C\frac{\mu_2}{n^2\mu_1^2}\le C'\frac{1}{\mu_1 n^2}.
\end{split}
\end{equation*}
The second inequality here follows by the substitution $w= x/\mu_2$ and the fact $$z^2\exp(-2z)\le Ce^{-z}.$$
Hence, \eqref{eq:deltabnorm} is proved.

By \eqref{eq:deltabnorm}, one finds easily that \eqref{eq:norm_deltab} is controlled by 
\[
\mathbb{E}\|\delta b\|^2\le
C \frac{1}{n^2}\sum_{\ell: \inf\{\mu: \mu\in S_{\ell}\}<2/n}|S_{\ell}|+C\frac{1}{n}\int_{2n^{-1}}^1\frac{1}{n^2\mu}d\mu
\]
and thus the result follows.
\end{proof}

\section{Discussion}\label{sec:discussion}

In this paper, we have provided the rigorous convergence proof for ROM in the slab geometry with isotropic scattering. One of the most important benefit of ROM is its ability to mitigate the ray effect, which is a long standing and significant problem in X-Y geometry and spatial 3D RTE simulations. Thus the extension of the current work to higher dimensional case would be of particular interest.  Similar convergence property can be observed numerically in the high dimensional cases with anisotropic scattering. We expect that the current framework of proving the convergence order can be extended to more complex cases but with more complex details about the properties of the operators.  This will be our future work. Moreover, it would be very interesting to investigate whether the framework of proving the improvement of convergence order can be extended to some other ray effect mitigating strategies.


\appendix

\section{An estimate of the error introduced by regularization}\label{app:regerr}

Recall the 1D RTE 
\begin{equation}\label{eq:1drte}
    \mu\partial_x \psi(x,\mu)+\sigma_{T}(x)\psi(x,
    \mu)=\lambda \sigma_{r}(x)\mathcal{I}(\psi)+q(x),
\end{equation}
and the regularized model
\begin{equation}\label{eq:1drtereg}
    \mu\partial_x \psi^{\delta}(x,\mu)+\sigma_{T}(x)\psi^{\delta}(x,
    \mu)=\lambda \sigma_{r}(x)\mathcal{I}^{\delta}(\psi^{\delta})+q(x),
\end{equation}
considered in this paper (with the same boundary conditions), where
\[
\mathcal{I}(\psi)=\fint_{S}\psi(x,\mu)d\mu,\quad \mathcal{I}^{\delta}(\psi)=\fint_{S^{\delta}}\psi(x,\mu)\,d\mu.
\]
We aim to estimate the error
\begin{gather}
v(x,\mu):=\psi^{\delta}(x,\mu)-\psi(x, \mu).
\end{gather}
It is easy to see that
\begin{gather}\label{eq:A4}
\mu\partial_x v(x,\mu)+\sigma_{T}(x)v(x, \mu)
=\lambda \sigma_{r}(x)\mathcal{I}^{\delta}(v)+\lambda\sigma_r f(x),
\end{gather}
with zero boundary conditions, where $f$ is the consistency error
\begin{gather}
f(x)=\mathcal{I}^{\delta}(\psi)-\mathcal{I}(\psi),
\end{gather}
which clearly goes to zero as $\delta\to 0$.
Define the corresponding mean density
\[
\phi=\mathcal{I}(\psi),\quad \phi^{\delta}=\mathcal{I}^{\delta}(\psi^{\delta}).
\]
Below, we use $\|\cdot\|$ to indicate the $L^2(\Omega; \sigma_T)$ norm. The main conclusion is the following.
\begin{proposition}
Consider the two models above, the error of the regularized model $v$ satisfies that
\begin{gather}
\|\mathcal{I}^{\delta}(v)\|\le \sup_{\mu}\|v(\cdot, \mu)\|\le \frac{\lambda}{1-\lambda}\|f\|
\end{gather}
and
\begin{gather}
\|\phi-\phi^{\delta}\|\le (\frac{\lambda}{1-\lambda}+1)\|f\|=\frac{1}{1-\lambda}\|f\|.
\end{gather}
\end{proposition}

\begin{proof}
For $\mu>0$, in \eqref{eq:A4}, multiplying $v$ and integrate with respect to $x$, one has
\[
\mu \frac{1}{2}v_{x_R}^2+\|v(\cdot, \mu)\|^2\le \lambda \|v(\cdot, \mu)\|\|I^{\delta}(v)\|+\|v(\cdot, \mu)\| \|f\|.
\]

Hence,
\[
\|v(\cdot, \mu)\|\le \lambda \|\mathcal{I}^{\delta}(v)\|+\lambda \|f\|
\]
The estimate for $\mu<0$ is similar and we skip.

Since
\[
\|\mathcal{I}^{\delta}(v)\|\le \fint_{S^{\delta}} \|v(\cdot, \mu)\|\,d\mu,
\]
one then concludes that
\[
\|\mathcal{I}^{\delta}(v)\|\le \frac{\lambda}{1-\lambda}\|f\|.
\]
Inserting this back, one then has
\[
\|v(\cdot, \mu)\|\le (\frac{\lambda^2}{1-\lambda}+\lambda)\|f\|=\frac{\lambda}{1-\lambda}\|f\|.
\]

For $\phi$ and $\phi^{\delta}$, using the definition, one then has
\[
\|\phi-\phi^{\delta}\|\le (\frac{\lambda}{1-\lambda}+1)\|f\|=\frac{1}{1-\lambda}\|f\|.
\]
\end{proof}
The estimate blows up as $\lambda\to 1$. This can be improved if one analyzes the diffusion regime carefully
such as in \cite{GolseandJin}, and we skip this.

\section{Detailed Derivation from 3D Radiative Transfer Equation to Slab Geometry Equation }\label{app:reduction}

The spatial-three-dimensional radiative transfer equation (RTE) with isotropic scattering is given by:
\begin{equation}\label{equ:3DRTE}
\mathbf{u} \cdot \nabla \psi(\mathbf{z}, \mathbf{u})+\sigma_{T}(\mathbf{z}) \psi(\mathbf{z}, \mathbf{u})=\frac{\sigma_{S}(\mathbf{z})}{4\pi}\int_{S} \psi(\mathbf{z}, \mathbf{u}') \mathrm{d} \mathbf{u}'+q(\mathbf{z}).
\end{equation}In Cartesian coordinates, the velocity direction  \(\mathbf{u}\) is defined using two angles:
\begin{itemize}\item Polar angle \(\beta\): The angle with respect to the x-axis, where \(\mu = \cos\beta\).\item Azimuthal angle \(\theta\): The projection angle in the y-z plane relative to the y-axis (\(0 \leq \theta < 2\pi\)), which describes the rotational orientation of the particle's direction within the transverse (y-z) plane.\end{itemize}
Thus, the velocity direction is expressed as:
$$
\b{u}=(\mu,\eta,\xi)=(\cos\beta, \sin\beta \cos\theta, \sin\beta \sin\theta)=(\mu, \sqrt{1-\mu^2}\cos\theta, \sqrt{1-\mu^2}\sin\theta).
$$

With \(\mathbf{z}=(x,y,z)\) and direction components \(\mathbf{u}=(\mu,\eta,\xi)\), the RTE in \eqref{equ:3DRTE} can be rewritten as:
\begin{equation}\label{B2}
\begin{aligned}
&\mu \frac{\partial \psi}{\partial x} + \eta \frac{\partial \psi}{\partial y} + \xi \frac{\partial \psi}{\partial z}+\sigma_{T}(x,y,z) \psi(x,y,z, \mu,\eta,\xi)\\
=&\frac{\sigma_{S}(x,y,z)}{4\pi}\int_{0}^{2\pi}\int_{-1}^{1} \psi(x,y,z, \mu',\eta',\xi'),d\mu' ,d\theta' +q(x,y,z).
\end{aligned}
\end{equation}We make the following assumptions:
\begin{enumerate}
\item The solution is periodic in the y and z variables with periods \([y_l,y_r]\) and \([z_l,z_r]\), respectively.
\item \(\sigma_T\), \(\sigma_S\), and q depend only on x and are uniform in y and z.
\end{enumerate}Define the probability density function:
\begin{equation}
\bar{\psi}(x,\mu)=\int_{y_l}^{y_r}\int_{z_l}^{z_r}\int_{0}^{2\pi}\psi\left(x,y,z,\mu, \sqrt{1-\mu^2}\cos\theta, \sqrt{1-\mu^2}\sin\theta\right) ,d\theta ,dz ,dy.
\end{equation}
This function \(\bar{\psi}(x,\mu)\) is independent of y, z, and \(\theta\). By integrating both sides of \eqref{B2} over \(y \in [y_l,y_r]\), \(z \in [z_l,z_r]\), and \(\theta \in [0,2\pi]\), we obtain:$$\big(\mu\frac{\partial\bar{\psi}}{\partial x}+\sigma_T(x)\big)\bar{\psi}=\frac{1}{2}\sigma_S(x)\int_{-1}^1\bar{\psi}(x,\mu')\,d\mu' + q(x),$$which is the RTE in slab geometry. Here, the derivative terms with respect to y and z  in \eqref{B2} vanish due to the periodic boundary conditions in y and z. Additionally, since \(\sigma_T\), \(\sigma_S\), and q depend only on x, we can derive the governing equation for \(\bar{\psi}\).

\bibliographystyle{siamplain}
\bibliography{references}
\end{document}